\documentclass{amsart}
\usepackage{amsmath, amssymb, amsfonts}
\usepackage{cite}
\usepackage{tikz-cd}
\usepackage{lineno}
\usepackage[colorlinks, linkcolor=blue, anchorcolor=blue, citecolor=red]{hyperref}
\numberwithin{equation}{section}
\newcommand{\catO}{\mathcal{O}}
\newcommand{\frakg}{\mathfrak{g}}
\newcommand{\frakh}{\mathfrak{h}}
\newcommand{\frakn}{\mathfrak{n}}
\newcommand{\frakb}{\mathfrak{b}}
\newcommand{\frakp}{\mathfrak{p}}
\newcommand{\frakq}{\mathfrak{q}}
\newcommand{\fraku}{\mathfrak{u}}
\newcommand{\calU}{\mathcal{U}}
\newcommand{\frakl}{\mathfrak{l}}
\newcommand{\fraka}{\mathfrak{a}}

\newcommand{\bbC}{\mathbb{C}}
\newcommand{\bbZ}{\mathbb{Z}}

\newcommand{\calR}{\mathcal{R}}
\newcommand{\calC}{\mathcal{C}}
\newcommand{\calD}{\mathcal{D}}

\newcommand{\calM}{\mathcal{M}}

\newcommand{\Hom}{\mathrm{Hom}}
\newcommand{\Dim}{\mathrm{dim}}
\newcommand{\Ext}{\mathrm{Ext}}
\newcommand{\End}{\mathrm{End}}
\newcommand{\Mod}{\mathrm{Mod}}
\newcommand{\Ker}{\mathrm{Ker}}
\newtheorem{theorem}{Theorem}[section]

\newtheorem{remark}[theorem]{Remark}
\newtheorem{corollary}[theorem]{Corollary}

\newtheorem{proposition}[theorem]{Proposition}

\begin{document}	
	\title[A note on parabolic Verma module homomorphisms]{A note on parabolic Verma module homomorphisms over Kac-Moody Algebras}
	\author{Xingpeng Liu }

	\address{Department of Mathematics, University of Science and Technology of China, Hefei, 230026, Anhui, P. R. China}
	\email{xpliu127@mail.ustc.edu.cn}


	\maketitle
\begin{abstract}
      We consider a general  parabolic category $\catO_{S}$ over symmetrizable Kac-Moody algebras. We introduce a reduction rule of hom-spaces between parabolic Verma modules over different Kac-Moody algebras which yield some applications on  parabolic Verma module homomorphisms.
	
	\end{abstract}	
	
     \section{Introduction}
     	Let $\frakg = \frakg(A)$ be a complex symmetrizable Kac-Moody algebra, where $A$ is an $n \times n$ generalized Cartan matrix. Let $I= \{1,2, \cdots, n \}$ be an index set. There exists a parabolic subalgebra of $\frakg$ associated to any subset $S$ of $I$. In this paper, we extend the parabolic category $\catO_{S}(\frakg)$ (or simply $\catO_{S}$) over $\frakg$ first studied by Rocha-Caridi and Wallach \cite{RW} to a general case where the parabolic subalgebra can be chosen for any subset $S$ of $I$ and we study certain modules of importance in the representation theory of Kac-Moody algebras, i.e., parabolic Verma modules $M_S(\lambda)$ (or generalized Verma modules) in  $\catO_{S}$.
     
        The parabolic category $\catO_{S}$ is a natural generalization of the well-known category $\catO$ introduced by Bernstein, Gelfand and Gelfand \cite{BGG} for the finite dimensional case and studied by Deodhar, Gabber and Kac for the infinite dimensional Kac-Moody case \cite{DGK}.  For the empty set $S$,  the parabolic category $\catO_{S}$ is just the category $\catO$. There are three types of modules in $\catO_{S}$ which play important roles in studying the structure of $\catO_{S}$, the simple modules $L(\lambda)$, parabolic Verma modules $M_S(\lambda)$ and their projective covers $P_S(\lambda)$. Although projective objects in parabolic categories $\catO_{S}$ may not exist in general, they do appear in certain truncated subcategories.  All modules of these three types are parametrized by the subset $P_S^+$ of $\frakh^*$ on a fixed Cartan subalgebra $\frakh$ of $\frakg$. In general,  these modules are linked by the BGG reciprocity of  parabolic version.
   
       From \cite{Soe} (or cf.  \cite{Fie1}) we have a block decomposition \[ \catO_S= \prod _{\Lambda \in P_S^+/\sim_S}\catO_{S, \Lambda}, \]
       where $\sim_S$ is the equivalence relation on $P_S^+$ defined in Subsection \ref{subsection4.1} and for each equivalence class $\Lambda \in P_S^+/\sim_S$ the block $\catO_{S, \Lambda}$ is indecomposable. Thanks to the BGG reciprocity of parabolic version, an equivalence class $\Lambda$ can be described  by  pairs $(\lambda, \mu) \in \frakh^* \times \frakh^*$ with nontrivial multiplicity $[M_S(\lambda): L(\mu)]$.  If $S$ is the empty set, then the equivalence class $\Lambda \in \frakh^*/\sim$ is determined by Verma and BGG's theorem (cf. e.g.  \cite{BGG,Verma,Hum}) for finite dimensional simple Lie algebras and by an analog of BGG's theorem \cite{KK} for general Kac-Moody algebras. 
       
       For finite dimensional simple Lie algebra, there is an extensive study on the hom-space between parabolic Verma modules.  For instance,  Lepowsky \cite{L} and Boe \cite{Boe} worked out many examples and gave precise sufficient and necessary conditions for a standard map to be zero; Zuckerman gave a useful duality theorem on hom-spaces of parabolic Verma modules ~\cite{BC,G}; Matumoto \cite{M} explored nonzero homomorphisms between two parabolic Verma modules of scalar type; Xiao \cite{WX} classified all first order leading weight vectors and determined corresponding hom-space between parabolic Verma modules. For more relevant researches on finite dimensional simple Lie algebras, one can refer to \cite{L, Irv,Hum}. However, there are few results for general Kac-Moody algebra cases. 
       Note that a basic difference between the cases of finite and infinite dimensional Kac-Moody algebras is that the property of  unique embeddings between Verma modules fails for infinite dimensional Kac-Moody algebras, namely, the dimension of hom-space
      \[ \mathrm{dim}_\bbC \mathrm{Hom}_{\catO_{S}}(M(\mu), M(\lambda)) \leq 1 \] is not true for  $\lambda, \mu \in \frakh^*$ in general \cite{KK}. Moreover, for infinite dimensional Kac-Moody algebras, their universal enveloping algebras are not (left) noetherian \cite{Berman}. 
    
   In this paper, our research concerns itself with the hom-space between parabolic Verma modules of a symmetrizable Kac-Moody algebra $\frakg$. First, we extend the parabolic category $\catO_{S}$ over $\frakg$ first studied by Rocha-Caridi and Wallach \cite{RW} to a general case in where the parabolic subalgebra can be chosen for any subset $S$ of $I$.  For any equivalence class $\Lambda \in P_S^+$, we introduce a reduction rule of hom-space between parabolic Verma modules, more precisely, we have the following theorem (i.e., Theorem \ref{Lmm}).
   \begin{theorem}~\label{main1}
   	 For any subset $S$ of $I$, any weights $\lambda, \mu$ in $P_S^+$,  let the subsets $S' \subset J$ of $I$ be defined in section ~\ref{Sec5}. Then we have a vector space isomorphism 
   	\[\mathrm{Hom}_{\catO_S(\frakg)}(M_S(\mu),M_S(\lambda)) \cong \mathrm{Hom}_{\catO_{S'}(\frakl_{J})}(V_{S'}(\mu),V_{S'}(\lambda)), \]
   	where $V_{S'}(\lambda)$ denotes the parabolic Verma module of $\frakl_{S'}$ in terms of $S'$ with the highest weight $\lambda$. 
   	\end{theorem}
   
   Moreover, we obtain that there exists the unique embedding  property for certain cases (i.e., Theorem ~\ref{cor7.3}) by using the reduction rule and the duality introduced in Subsection ~\ref{subsection42}.
   
    \begin{theorem}
   	For $\lambda, \mu \in P_S^+$, let $J:=\mathrm{supp}(\lambda - \mu) \subset I$ be defined in Section \ref{Sec5}. If the height and the $S$-height of $\lambda-\mu$ are equal, i.e.,  $\mathrm{ht}_S(\lambda -\mu)= \mathrm{ht}(\lambda - \mu)$ and $\lambda$ is of $J$-positive level (or of $J$-negative level), then 
   	  \[   
   		\mathrm{dim}_\bbC \mathrm{Hom}_{\catO_S}(M_S(\mu), M_S(\lambda)) \leq 1,
    \]
    where the $S$-height and the height of any weights are defined in Subsection \ref{subsection42}.
   \end{theorem}	
   
    In addition, for $S$ of finite type (cf. the end of Section \ref{sec2}) and an equivalence class $\Lambda$ of negative level, we introduce a finite truncation $\calC$ of $\catO_{S, \Lambda}$, which is a subcategory of $\catO_{S, \Lambda}$, then we determine a quotient category $\calD$ at the positive level, which is Ringel dual to $\calC$. The parabolic Verma modules $M_S(\lambda)$ correspond to smaller standard modules $A(\lambda')$ for some related weight $\lambda'$ under the Ringel dual. Then we have a vector space isomorphism (cf. Theorem \ref{thmredu})
   \[\Hom_{\catO_S}(M_S(\mu), M_S(\lambda)) \cong \Hom_{\calD}(A(\lambda'), A(\mu')).\]

   The paper is arranged as follows. In Section \ref{sec2}, we recall some basic notations and properties of symmetrizable Kac-Moody algebras. In Section \ref{BD}, we give a more general definition of parabolic category $\catO_S$ of $\frakg$. Then  certain modules in $\catO_{S}$, the block decomposition of $\catO_{S}$ and tilting functors will be discussed. In Section \ref{Sec5}, a reduction rule of hom-space between parabolic Verma modules and some applications will be talked about.
   
   \textsc{Conventions}. Throughout the paper, we denote the set of complex numbers, integers  by $\bbC$, $\bbZ$ respectively. Set $\bbZ_{\geq 0}$ (resp. $\bbZ_{\leq 0}$)  for the set of non-negative integers (resp. non-positive integers).  Let $I=\{1,2, \cdots, n \}$ be an index set.

	 \section{Notations and preliminaries}\label{sec2}

	First, let us recall some basic notations and properties of symmetrizable Kac-Moody algebras  based on \cite{Kac,Hum}.
	
	Fix a symmetrizable generalized Cartan matrix $A=(a_{ij})_{i, j \in I}$. Let $\frakg$ be the symmetrizable Kac-Moody algebra associated to $A$. Let $\frakh \subset \frakb$ be its Cartan subalgebra, Borel subalgebra respectively. There exists a root space decomposition of $\frakg$ with respect to $\frakh$:
	\[ \frakg = \frakh \oplus (\bigoplus_{\alpha \in \Delta} \frakg_\alpha), \]
	where $\Delta \subset \frakh^*$ is the root system of $\frakg$ relative to $\frakh$. The choice of $\frakb$ determines a set of positive roots $\Delta^+$ and a disjoint union $\Delta=\Delta^+ \cup \Delta^-$, where $\Delta^-=-\Delta^+$. Moreover, let $\Pi=\{\alpha_i, i \in I \} \subset \Delta^+$ be the set of simple roots. The corresponding set of coroots is denoted by $\Pi^\vee=\{\alpha_i^\vee, i \in I \}$, such that $ \langle \alpha_i, \alpha_j^\vee \rangle = a_{ji}$, where $ \langle \cdot, \cdot \rangle : \frakh^* \times \frakh \rightarrow \bbC$ is the canonical pairing. Set $Q= \bbZ \Pi$. Fix a non-degenerate symmetric bilinear form $(\cdot, \cdot)$ on $\frakh^*$ such that $(\alpha_i, \alpha_i)$ is a positive rational number and $\langle \lambda, \alpha_i^\vee \rangle = 2(\lambda, \alpha_i)/(\alpha_i, \alpha_i)$, for any $\lambda \in \frakh^*$, $i \in I$. Let 
	\[  \Delta^{im}=\{\alpha \in \Delta; (\alpha, \alpha)\leq 0 \}, \quad \Delta^{re}=\{\alpha \in \Delta; (\alpha, \alpha) > 0 \} \]
	  be the set of imaginary roots and the set of real roots of $\frakg$ respectively. In addition, $\frakh^*$ has a natural partial order $\leq$ defined as $\lambda \leq \mu$ if and only if $\mu -\lambda \in \bbZ_{\geq 0}\Delta^+$. 
	
	Let $S $ be any subset of  $I$. We denote $\Delta_S = \Delta \cap \Sigma_{i \in S}\bbZ\alpha_i$ and $\Delta_S^{\pm}= \Delta^\pm \cap \Delta_S$. Define the subalgebras $\frakh_S$, $\frakn_S^{\pm}$ and $\fraku_S^{\pm}$ of $\frakg$ (sometimes write $\frakn_S^+$ and $\fraku_S^+$ as $\frakn_S$ and $\fraku_S$ respectively for simplicity) as follows:
	\[ \frakh_S = \sum_{i \in S} \bbC \alpha_i^\vee, \quad
	\frakn_S^\pm = \bigoplus_{\alpha \in \Delta_S^\pm}\frakg_\alpha, \quad
	\fraku_S^\pm = \bigoplus_{\alpha \in \Delta^\pm \setminus \Delta_S^\pm}\frakg_\alpha.
	\]	
	Set 
	\[
	\frakg_S = \frakh_S \oplus \frakn_S \oplus \frakn^-_S, \quad 
	\frakl_S = \frakg_S + \frakh, \quad
	\frakp_S = \frakl_S \oplus \fraku_S.
	\]
	Obviously, those are all subalgebras of $\frakg$. We call such \textit{$\frakp_S$ the parabolic subalgebra} of $\frakg$ with respect to the subset $S$. In particular, $\frakp_S=\frakb$ if $S= \emptyset$. 
	
	Let $r_i$ be the \textit{reflection} on $\frakh^*$ determined by $\alpha_i, i \in I$, i.e., for any $\lambda \in \frakh^*$, $r_i(\lambda) = \lambda - \langle \lambda, \alpha_i^\vee \rangle \alpha_i$. Let $W$ be the \textit{Weyl group} of the Kac-Moody algebra $\frakg$ generated by $r_i, i \in I$. By the definition of real roots, we have $\Delta^{re}= W \Pi$ and $r_\alpha = wr_iw^{-1}$ is also a reflection where $\alpha = w(\alpha_i)$. 	Fix an element $\rho \in \frakh^*$ such that $\langle \rho, \alpha^\vee_i \rangle = 1$ for all $i \in I$. Obviously, such an element in $\frakh^*$ is not unique. 
	
	Note that $\frakl_S$ is essentially a symmetrizable Kac-Moody Lie algebra  determined by the generalized Cartan matrix $A_S=(a_{ij})_{i,j \in S}$. In fact, set $\frakh^S := \{h \in \frakh ; \langle\alpha_j, h \rangle =0, j \in S \}$ and $\frakh_1$ a complementary subspace of $\frakh^S + \frakh_S$ in $\frakh$. Let $\frakh_S' = \frakh_S \oplus \frakh_1$, then the subalgebra of $\frakl_S$ defined as  
	\[ \frakh_S' \oplus (\bigoplus_{\alpha \in \Delta_S} \frakg_\alpha) \]
	is isomorphic to the symmetrizable Kac-Moody algebra  $\frakg(A_S)$. As a result, the root system $\Delta_S$ is a subsystem of $\Delta$, the Weyl group $W_S$ of $\frakl_S$ is a subgroup of $W$ generated by $\{r_i, i \in S \}$.   We may endow $\frakl_S$ with a non-degenerate symmetric bilinear form in such a way that this form agrees with $(\cdot, \cdot)$ when restricted to $\frakh \times \frakh$. The set of real roots $\Delta_S^{re}=\{\alpha \in \Delta_S; (\alpha, \alpha)>0 \} $ and the set of imaginary roots $\Delta_S^{im}=\{\alpha \in \Delta_S; (\alpha, \alpha)\leq 0 \}$ of $\frakl_S$(or $ \frakg(A_S)$) do make sense in this way. Moreover,  $S$ is of \textit{finite, affine or indefinite type} if $\frakl_S$ as a Kac-Moody algebra has the corresponding type.

	 \section{Categories and modules}\label{BD}
	 
	 Fix a subset $S$ of the index set $I$. In this section, a general definition of parabolic categories over  $\frakg$ will be given and basic notions and facts on certain prominent objects in parabolic categories $\catO_S$  will be discussed. 
	 \subsection{Parabolic categories}
	Suppose that $\fraka$ is a Lie subalgebra of $\frakg$ containing $\frakh$.  For any $\fraka$-module $M$ and $\mu \in \frakh^*$, set
	$M_\mu = \{v \in M ; h.v=\mu(h)v, \forall h \in \frakh \} $ the \textit{weight space} of $M$ with respect to the weight $\mu$. If $M = \oplus M_\mu$, then we call $M$ a \textit{weight $\fraka$-module}. Set 
	 $ P(M)= \{\mu \in \frakh^* ; M_\mu \neq 0 \}$ the set of weights of $M$.
	  Let $\calU(\fraka)$ be the universal enveloping algebra of $\fraka$. Note that $\calU(\fraka)$ can be seen as  a weight module by adjoint $\fraka$-action and we write $\calU(\fraka)_{\beta}$ for the $\beta$-weight space of this module. Recall the integrable $\frakg$-module defined in \cite[Sec.3.6]{Kac}, similarly, we can consider the integrable modules over the subalgebra $\frakl_S$.
	  
	   The \textit{parabolic category} $\catO_S(\frakg)$ (or simply $\catO_S$) over $\frakg$ is the full subcategory of the weight $\frakg$-module category consisting of all weight $\frakg$-modules $M$  which are locally $\frakb$-finite and are integrable as weight $\frakl_S$-modules.
	  
	  Here we extend the parabolic category $\catO_S$ over Kac-Moody algebras which was first studied by Rocha-Caridi and Wallach \cite{RW} to a general case in which the parabolic subalgebra can be chosen for any subset $S$ of $I$. Obviously, $M \in \catO_S$ is locally $\fraku_S$-nilpotent by the above definition and is $\frakl_S$-semisimple by Proposition \ref{prop32} below.
	  
	  \begin{remark} 
	  	There are two extreme cases. If  $S$ is the empty set, then $\catO_S$ is exactly the BGG category $\catO$ over $\frakg$. If  $S$ is the whole index set, then any $\frakg$-module in $\catO_S$ is integrable, therefore, is semisimple. We call it the integrable parabolic category over $\frakg$. Denote such $\catO_S$ by $\catO_{int}$.
	  \end{remark}
	  
	  Note that any module $M \in \catO_S$ lies in the integrable parabolic category $\catO_{int}(\frakl_S)$. Hence we can easily deduce the following result by the complete reducibility theorem \cite[Theorem 10.7]{Kac}.
	  \begin{proposition} \label{prop32}
	  	Any module $M$ in the parabolic category $\catO_S$ is a completely reducible $\frakl_S$-module, i.e., $M$ as an $\frakl_S$-module is isomorphic to a direct sum of irreducible integrable highest weight $\frakl_S$-modules.
	  \end{proposition}	
	  From Proposition \ref{prop32}  all highest weight modules in $\catO_S$ have highest weights belonging to a subset $P_S^+$ of $\frakh^*$:
	  \[ P^+_S = \{\lambda \in \frakh^*; \langle\lambda, \alpha_i^\vee \rangle \in \bbZ_{\geq 0}, \forall i\in S \}.\]
	  Conversely, simple modules  in $\catO_S$ can be parametrized by this set $P_S^+$. We denote the simple module relative to $\lambda \in P_S^+$ by $L(\lambda)$, which is the unique simple quotient of any highest weight module with highest weight $\lambda$.

	Let the elements $e_i \in \frakg_{\alpha_i}, f_i \in \frakg_{-\alpha_i}, i \in I$ be the \textit{Chevalley generators} of $\frakg$. Consider the standard anti-involution $\sigma: \frakg \rightarrow \frakg$ such that $\sigma(e_i) = f_i$, $\sigma(f_i)=e_i$, for $i \in I$ and $\sigma|_\frakh = \mathrm{id}_\frakh$.  Let $D$ be the  \textit{duality functor} on $\catO$ mapping each module $M = \oplus_{\lambda \in \frakh^*}M_\lambda$ in $ \catO$ to its contragredient dual $DM = \oplus_{\lambda \in \frakh^*}(M_\lambda)^*$ in $\catO$, where the action of $\frakg$ on $(M_\lambda)^*$ is defined as $ (x\cdot \phi)(v) = \phi(\sigma(x)\cdot v) $
	for $v \in M, x \in \frakg$ and $\phi \in (M_\lambda)^*$. The duality functor $D$ is contravariant and exact. For any $M \in \catO$ with finite dimensional weight spaces, there is a natural isomorphism $M\cong D(D(M)) $. Moreover, $D$ preserves each parabolic category $\catO_S$.
	
	Let $i_S$ be the canonical inclusion $\catO_S \subset \catO$. It is a fully faithful exact functor and it admits a left adjoint $\tau_S $ which maps each object in $\catO$ to its maximal quotient in $\catO_S$. We call $\tau_S$ a \textit{parabolic truncation functor.}  The right adjoint functor of $i_S$ is  the conjugate $D\circ\tau_S\circ D $ of $\tau_S$ by $D$, which maps each object $M$ in $\catO$ to its maximal submodule in $\catO_S$.

    \subsection{Parabolic Verma modules}
   For each $\lambda \in P_S^+$, let $L_S(\lambda)$ be the integrable simple  weight $\frakl_S$-module of highest weight $\lambda$, which also gives a simple weight $\frakp_S$-module by trivial $\fraku_S$-action. A \textit{parabolic Verma module} $M_S(\lambda)$ is defined as an induced $\frakg$-module by  $\frakp_S$-module $L_S(\lambda)$ as follows:
   \[ M_S(\lambda):= \calU(\frakg) \otimes_{\calU(\frakp_{S})}L_S(\lambda). \]
   It has a unique simple quotient $L(\lambda)$. Obviously, $M_S(\lambda) \in \catO_S$. In particular, if $S=\emptyset$, then $M(\lambda) := M_S(\lambda)$ is just a \textit{Verma module}. Moreover, applying the parabolic truncation functor $\tau_S$ on $M(\lambda)$ we get the parabolic Verma module $M_S(\lambda)$.
   
   A difference between the cases of finite and infinite dimensional Kac-Moody algebras is that the unique embedding theorem fails for infinite dimensional Kac-Moody algebras, namely, the dimension of hom-space
   \[ \mathrm{dim}_\bbC \mathrm{Hom}_{\catO}(M(\mu), M(\lambda)) \leq 1 \] is not true for  $\lambda, \mu \in \frakh^*$ in general. 
    However, \cite{KK} pointed out that we still have an analogous version of BGG theorem.
   \begin{theorem} \label{thm3.3}
   	The space $\mathrm{Hom}_{\catO}(M(\mu), M(\lambda))$ is not trivial, or equivalently, the multiplicity $[M(\lambda): L(\mu)]$ is not zero iff there exist a collection $\{\lambda_0, \cdots, \lambda_t \}$ of weights in $\frakh^*$,  a collection $\{\beta_1, \cdots, \beta_t \}$ of positive roots in $\Delta^+$ and  a collection $\{n_1, \cdots, n_t \}$ of positive integers satisfying 
   	\[\lambda_0= \lambda, \quad \lambda_t=\mu, \quad \lambda_j = \lambda_{j-1}-n_j \beta_j \text{ and } 2(\lambda_{j-1}+\rho, \beta_j) = n_j(\beta_j, \beta_j) \]
   	for any $j = 1,2, \cdots, t$.
   	\end{theorem}
   
   In general, for a $\frakg$-module $N$ with finite dimensional weight spaces in $\catO_S$, we may consider the multiplicity $[N : L(\lambda)]$ of $L(\lambda)$ in $N$ for any $\lambda \in P_S^+$. Moreover, we have 
   \[\mathrm{dim}_\bbC \mathrm{Hom}_{\catO_{S}}(M_S(\lambda), N) \leq [N: L(\lambda)]. \]
    In particular, the dimension of the space of $\frakg$-module homomorphisms between parabolic Verma modules is finite.  Unfortunately, it is still difficult to determine although the dimension is finite.  

    \subsection{Projective objects}\label{projective}
     Although projective objects in parabolic categories may not exist in general, they  appear in certain truncated subcategories.  
    
    For any weight $\lambda \in \frakh^*$,  a \textit{truncated subcategory} $\catO_S^{\leqslant \lambda}$ consists of all modules $M$ in $\catO_S$ with weights $\mu\leq \lambda$, for all $\mu \in P(M)$.
    The functor $\catO_S \rightarrow \catO_S^{\leqslant \lambda}$ which maps each module $M \in \catO_S $ to the quotient $M/ \sum_{\mu \nleq \lambda}\calU(\frakg)M_\mu$ is left adjoint to the inclusion functor (see, e.g., \cite{Fie1}).
    
    Note that $\fraku_S$ is a (nil-radical) ideal of $\frakp_S$, $\calU(\fraku_S)$ is the $\frakp_S$-module by adjoint action with weight space decomposition $\calU(\fraku_S)= \oplus\gamma\calU(\fraku_S)_\gamma$. For a weight $\nu \in \frakh^*$, let $\calU(\fraku_S)^{\nleqslant \nu} $  be the $\calU(\frakp_S)$-submodule of $\calU(\fraku_S)$ generated by the space
    $ \oplus_{\gamma \nleq \nu}\calU(\fraku_S)_\gamma $. Denote by 
    $ \calU(\fraku_S)^{\leqslant \nu}$ the quotient module $\calU(\fraku_S)/\calU(\fraku_S)^{\nleqslant \nu} $. 
    
   For any $\lambda, \mu \in \frakh^*$, $\mu \leq \lambda$, set $\eta = \lambda - \mu$. If $\mu \in P^+_S$, then we consider the following  induced $\frakg$-module,
   \begin{align} \label{proj}
   Q_S^{\leqslant \lambda}(\mu): = \calU(\frakg)\otimes_{\calU(\frakp_S)}(\calU(\fraku_S)^{\leqslant \eta}\otimes_\bbC L_S(\mu)) . 
   \end{align}
   
   As in  \cite{RW}, for any $M \in \catO_{S}^{\leqslant \lambda}$, there is a natural isomorphism $\Hom_\frakg(Q_S^{\leqslant \lambda}(\mu), M) \rightarrow \Hom_{\frakl_S}(L_S(\mu), M)$. Due to Proposition \ref{prop32} the module $ Q_S^{\leqslant \lambda}(\mu) $ is a projective $\frakg$-module in the category $\catO_S^{\leqslant\lambda}$. Moreover,  $ Q_S^{\leqslant \lambda}(\mu) $ has a \textit{finite parabolic Verma flag}, i.e., a finite filtration with all its subquotients being isomorphic to parabolic Verma modules. More precisely, it has the parabolic Verma module $M_S(\mu)$ as its subquotient with multiplicity 1 and all other subquotients are isomorphic to certain parabolic Verma modules $M_S(\gamma)$ with highest weights $\gamma \geq \mu$. Let $ P_S^{\leqslant \lambda}(\mu) $ be the  indecomposable direct summand of $Q_S^{\leqslant \lambda}(\mu)$ which contains $M_S(\mu)$ as its subquotient. The module $ P_S^{\leqslant \lambda}(\mu) $ is exactly the projective cover of $L(\mu)$ in $\catO_S^{\leqslant\lambda}$ (see, for more details, e.g. \cite{RW} or \cite{Fie1}).
   
     Let $\calM_S \subset \catO_S$ be the full subcategory of modules which admit a finite parabolic Verma flag. $\calM_S$ is closed under taking summands and direct sums. So it is an exact subcategory. For arbitrary $M \in \calM_S$, denote the multiplicity of $M_S(\gamma)$ in $M$ by  $(M: M_S(\gamma))$, it is independent of the choice of finite parabolic Verma flags.
   
   We are now able to give the generalization of the BGG reciprocity formula.
   \begin{theorem} \label{thm3.4}
   	Let $\lambda, \mu \in \frakh^*$, $\mu \leq \lambda$. If $\mu \in P_S^+$, then the projective cover $ P_S^{\leqslant \lambda}(\mu) $ of the simple module $L(\mu)$ exists in $\catO_S^{\leqslant\lambda}$ and 
   	\[(P_S^{\leqslant \lambda}(\mu) : M_S(\gamma))=[M_S(\gamma): L(\mu)] \]
   	for any $M_S(\gamma)$ in  $\catO_S^{\leqslant\lambda}$.
   	\end{theorem}
   \begin{proof}
   	The module $ P_S^{\leqslant \lambda}(\mu) $ is a direct summand of $Q_S^{\leqslant \lambda}(\mu)$ and $Q_S^{\leqslant \lambda}(\mu) \in \calM_S$ as above. So $ P_S^{\leqslant \lambda}(\mu) $ has a finite parabolic Verma flag, and following the proof of \cite[Theorem 6.4]{RW}, we get the formula as desired.
   	\end{proof}
    Note that all finitely generated projective modules in $\catO_S^{\leqslant\lambda}$ have finite parabolic Verma flags.

    We next consider the block decomposition of the parabolic category $\catO_S$	motivated by \cite{Soe} and \cite{Fie1}. Then we give a duality between finite truncations of non-critical blocks.
   \subsection{Block decomposition} \label{subsection4.1}
   Fix a subset $S$ of $I$. Let $\sim_S$ be the equivalence relation on $P_S^+$ generated by $\mu \sim_S \gamma$ if there exists a weight $\lambda \in \frakh^*$ such that the projective cover $P_S^{\leqslant \lambda}(\mu)$ has $M_S(\gamma)$ as a subquotient, equivalently, by Theorem \ref{thm3.4}, if the multiplicity $[M_S(\gamma): L(\mu)]$ is not zero. If $S$ is empty, we denote the equivalence relation $\sim_S$ by $\sim$.
   
   For any union of equivalence classes $\Theta $ in $ P_S^+ / \sim_S$, we consider $\catO_{S, \Theta}$ as the subcategory generated by all $P_S^{\leqslant \lambda}(\mu)$ for $\mu \in \Theta$ and $\lambda \in \frakh^*$. Obviously, $\catO_{S, \Theta} \subset \catO_{S}$ is a (full) abelian subcategory of modules $M$ such that all highest weights of subquotients of $M$ lie in $\Theta$. For any $\frakg$-module $M$ in $\catO_{S}$, define $M_\Theta$ as the maximal submodule of $M$ in $\catO_{S, \Theta}$.
   
  Hence we have the \textit{block decomposition} of $\catO_{S}$ in the following sense.
  
  \begin{proposition}
  	The functor 
  	\[\prod_{\Lambda \in P_S^+/ \sim_S} \catO_{S, \Lambda}  \longrightarrow \catO_{S}, \{M_\Lambda \}_\Lambda \longmapsto \oplus M_\Lambda \]
  	is an equivalence of categories.
  	
  	\end{proposition}
	\begin{proof}
		By \cite{Fie1}, we can define the inverse functor $N \mapsto \{N_\Lambda \}$, where the submodule $N_\Lambda$ defined as above is actually generated by the images of all morphisms $P_S^{\leqslant \lambda}(\mu) \rightarrow N$ for any $\mu \in \Lambda$ and $\lambda \in \frakh^*$.
	\end{proof}	
	Each subcategory $\catO_{S, \Lambda}$ for $\Lambda \in P_S^+/ \sim_S$ is indecomposable, and we call it a \textit{block} of $\catO_S$ associated to $\Lambda$. 
   
	The equivalence relation $\sim$ on $\frakh^*$, by Theorem \ref{thm3.3}, is generated by the pairs $(\lambda, \mu) \in \frakh^* \times \frakh^*$ that satisfy $2(\lambda + \rho, \alpha)= n(\alpha, \alpha)$ and $\lambda - \mu =n\alpha$ for $n \in \bbZ_{\geq 0}, \alpha \in \Delta$.
	
	Following \cite{KT1}, for arbitrary $\lambda \in \frakh^*$, we define 
	\[\Delta(\lambda)=\{\alpha \in \Delta^{re}; \langle \lambda, \alpha^\vee \rangle \in \bbZ \}, \]
	with its associated subgroup $W(\lambda)$ of $W$ generated by all reflections $r_\alpha$ for $\alpha \in \Delta(\lambda)$.
	
	 For $\lambda \in \frakh^*$ and $w \in W$, we define a shifted action of $W$ by $w\cdot \lambda= w(\lambda + \rho)-\rho$, which is independent on the choice of $\rho$. Next we consider the equivalence classes $\Lambda$ in $\frakh^*/\sim$ which can be determined by the Weyl group $W$.
	 
	 Let $\Lambda \in \frakh^* / \sim$, define
	 \[\Delta(\Lambda)= \{\alpha \in \Delta ; 2(\lambda + \rho, \alpha) \in \bbZ(\alpha, \alpha) \text{ for some } \lambda \in \Lambda \}. \]
	
	We call an equivalence class $\Lambda \in \frakh^*/\sim$ \textit{critical} if the set $\Delta(\Lambda) \cap \Delta^{im} \neq \emptyset$. Otherwise, $\Lambda$ is called \textit{non-critical}.
	
	For any non-critical equivalence class $\Lambda \in \frakh^*/\sim$, we call the subgroup $W(\Lambda)$ of $W$ generated by the reflections $r_\alpha$ for all $\alpha \in \Delta(\Lambda)$ the \textit{integral Weyl group} with respect to $\Lambda$.
	
	A weight $\lambda \in \frakh^*$ is called \textit{dominant} if $\langle \lambda + \rho, \alpha^\vee \rangle \notin \bbZ_{<0} $ for all positive real roots $\alpha$ in $\Delta^{re}$. Similarly, $\lambda$ is called \textit{anti-dominant} if $\langle \lambda + \rho, \alpha^\vee \rangle \notin \bbZ_{>0} $ for all positive real roots $\alpha$ in $\Delta^{re}$. Denote the set of dominant (resp. antidominant) weights by $\frakh^{*+}$ (resp. $\frakh^{*-}$).

	\begin{proposition}  \label{prop42}
		Let $\Lambda \in \frakh^*/\sim$ and $\mu \in  \frakh^*$.
		\begin{enumerate}
			\item $|(W(\mu)\cdot\mu) \cap \frakh^{*\pm}| \leq 1$.
			
			\item If $\Lambda$ is non-critical, then $\Lambda = W(\lambda)\cdot \lambda$ for any $\lambda \in \Lambda$.
			\end{enumerate}
		
	\end{proposition}	
	\begin{proof}
	 The first part follows from \cite[Proposition 3.12]{Kac}. If $\Lambda$ is non-critical, it is clear that $\Lambda = W(\Lambda)\cdot \lambda$. Since $\Delta(\Lambda)= \Delta(\lambda)$, we have $W(\Lambda)= W(\lambda)$ for any $\lambda \in \Lambda$. 
	\end{proof}

  We say a non-critical equivalence class $\Lambda \in \frakh^*/\sim$ is of \textit{positive} (resp. \textit{negative}) \textit{level} if it contains a (unique) dominant (resp. anti-dominant) weight. 
   
   A dominant weight in $\Lambda$ is maximal in the poset $(\Lambda, \leq)$. So there are enough projective objects in the block $\catO_\Lambda$ of positive level. 
   
   \subsection{Tilting functors} \label{subsection42}
   Let us recall the tilting functors introduced by Soergel in \cite{Soe}, which builds a connection between blocks of positive and negative levels.
    
    For any root $\alpha \in \Delta$, suppose that $\alpha = \sum_{i \in I} k_i \alpha_i $ for some integers $k_i$. We call the integer
   $\textrm{ht}_S \alpha := \sum_{i \in I \backslash S} k_i $ the \textit{S-height} of $\alpha$. 
   
   Each parabolic subalgebra $\frakp_S$ gives a $\bbZ$-grading  $  \frakg = \oplus \frakg_i $ of $\frakg$ where
   \[ \frakg_0=\frakl_S, \quad \frakg_i=\bigoplus_{\mathrm{ht}_S(\alpha)=i } \frakg_\alpha, \quad i \in \bbZ \text{ and } i \neq 0.\]
    In particular, if $S=\emptyset$, we say the $\bbZ$-grading of $\frakg$ defined as above is a \textit{principal $\bbZ$-grading} and the $S$-height $\textrm{ht}_S \alpha$ of $\alpha$ is just the height $\textrm{ht} \alpha$ of $\alpha$. Note that $\frakg_0, \frakg_1$ and $\frakg_{-1}$ generate the whole Lie algebra $\frakg$. If $S$ is of finite type, then all homogeneous spaces $\frakg_i$ are finite dimensional for all $ i \in \bbZ$.

    Let $S$ be of finite type. We now construct a tilting functor on the subcategory $\calM_S$ of $\catO_S$.
    
     The grading of $\frakg$ given by the parabolic subalgebra $\frakp_S$ implies that $\frakg_0= \frakl_S$. Note that $\frakl_S$ is a (finite dimensional) reductive Lie algebra. Let $\rho_S$ be the half sum of all elements in $\Delta^+_S$.  Then $\gamma_S := 2\rho-2\rho_S$ is a \textit{semi-infinite character} for $\frakg$ defined in \cite{Soe} (or cf. \cite[Chapter 7]{IK}).  
     There is a `semi-infinite' analogue of the universal enveloping algebra $\calU(\frakg)$, which is the semi-regular $\calU(\frakg)$-bimodule $S_{\gamma_S}(\frakg)$ associated to a semi-infinite character $\gamma_S$.
    For any  $\frakg$-module $M$ in $\catO_{S}$, we consider its antipode dual $M^\circledast= \oplus_{\lambda \in \frakh^*}(M_\lambda)^*$ with the antipode action $x.f(v) = -f(x.v), x \in \frakg, v \in M, f \in M^\circledast$. Then the tilting functor 
  \begin{equation}
  	\begin{array}{cccl}
  		t_S: & \calM_S & \longrightarrow & \calM_S^{opp} \\
  		& M  & \longmapsto & (S_{\gamma_S}(\frakg) \otimes_{\calU(\frakg)} M)^\circledast
  	\end{array}  	
  \end{equation}
   is an equivalence of exact categories (cf. \cite[Cor. 2.3]{Soe}, \cite[Chapter 7]{IK}).
   
    For any $\lambda \in \frakh^*$, $\mu \in P_S^+$, $\lambda \leq \mu$, there exists a unique indecomposable object $T_S^{\geqslant \lambda}(\mu)$ in $\catO_S$ \cite[Proposition 5.6]{Soe} satisfying 
   \begin{enumerate}
   	\item $\Ext^1_{\catO_{S}}(M_S(\nu), T_S^{\geqslant \lambda}(\mu))=0$ for all $\nu \geq \lambda$.
   	\item There is an inclusion $M_S(\mu) \hookrightarrow T_S^{\geqslant \lambda}(\mu)$ and its cokernel has a finite parabolic Verma flag with the only subquotients $M_S(\gamma)$ of highest weights $\lambda \leq \gamma < \mu$. 
   \end{enumerate}
   For any equivalence class $\Lambda \in \frakh^*/\sim$ of negative level and $\mu \in \Lambda_S :=\Lambda \cap P_S^+$, there always exists an indecomposable object $T_S(\mu) \in \catO_{S}$ which has a finite parabolic Verma flag with the only subquotients $M_S(\gamma)$ of highest weights $\gamma \leq \mu$ and  $\Ext^1_{\catO_{S}}(M_S(\nu), T_S(\mu))=0$ for all $\nu \in P_S^+$. Equivalently, both $T_S(\mu)$ and its dual $D(T_S(\mu))$ have a finite parabolic Verma flag, i.e., $T_S(\mu)$ is an indecomposable  \textit{tilting module}. 
  
   Let $w_S$ be the longest element in the finite Weyl group $W_S$.  For any weights $\lambda \in \frakh^*$, $\mu \in P_S^+$, $\lambda \geq \mu$, set $\lambda' = -\gamma_S-w_S\lambda$ and $\mu' = -\gamma_S-w_S\mu$, then $\lambda' \leq \mu'$. The tilting functor $t_S$ maps the parabolic Verma module $M_S(\mu)$ to the parabolic Verma module $M_S(\mu')$ and the module $P_S^{\leqslant \lambda}(\mu) $ to the indecomposable tilting module $ T_S^{\geqslant \lambda'}(\mu') $ (cf. \cite{Soe}). Actually, this is a version of Ringel equivalence for the parabolic category $\catO_{S}$.
   
   	We denote  the subcategory $\catO_{S, \Lambda_S}$ defined in the Subsection \ref{subsection4.1} by $\catO_{S, \Lambda}$ for convenience. Note that all subquotients of modules in $\catO_{S, \Lambda}$ have highest weights belonging to the set $\Lambda_S$ and $\catO_{S, \Lambda}$ is not a block of $\catO_{S}$ in general. Set $\calM_{S, \Lambda}= \catO_{S, \Lambda} \cap \calM_S$ (drop $S$ off if it is empty). 
   
   \begin{proposition}\label{Prop43}
   	Fix a subset $S$ of finite type. Let a non-critical equivalence class $\Lambda \in \frakh^*/ \sim$ be of positive level. Then $\Lambda' := -2\rho-\Lambda \in  \frakh^*/ \sim$ is a non-critical equivalence class of negative level. Restricted to $\calM_{S, \Lambda}$, the tilting functor 
   	\[ 	t_S:  \calM_{S, \Lambda}  \longrightarrow  \calM_{S, \Lambda'}^{opp} \]
   	is an equivalence of exact categories. $t_S$ maps a parabolic Verma module  $M_S(\mu)$ to $M_S(\mu')$ and the corresponding projective cover $P_S(\mu)$ to the indecomposable tilting module $T_S(\mu')$, where $\mu'=w_S\cdot(-2\rho-\mu)$.
   	\end{proposition}
   \begin{proof}
   	For a dominant weight $\lambda \in \Lambda$, it is easy to check that $-2\rho-\lambda$ is an anti-dominant weight, then by Proposition  \ref{prop42} we get $\Lambda' = W(\lambda)\cdot (-2\rho-\lambda)$ and  $\Lambda'$ is of negative level. Note that $w_S\cdot \mu = w_S\mu -2\rho_S$ and $w(\mu +\rho)-\rho = w(\mu +\rho_S)-\rho_S $ for all $w \in W_S$ and $\mu \in \frakh^*$. So we have 
   	\begin{equation*}
   		\begin{array}{ll}
   		 -\gamma_S -w_S\mu &= -2\rho -w_S\cdot \mu \\
   		 &= -\rho - w_S(\mu +\rho) \\
   		 &=-\rho + w_S(-2\rho - \mu +\rho)\\
   		 &= w_S\cdot (-2\rho -\mu). 
   		\end{array}  	
   	\end{equation*}
   Therefore, the dual set $-\gamma_S- w_S\Lambda_S$ is exactly the intersection of $\Lambda'$ and $P_S^+$.
    	\end{proof}
   Then we have 
   \begin{corollary}~\label{cor3.8}
   If $\Lambda$ is defined as the Proposition \ref{Prop43} above, then for any $\lambda, \mu \in \Lambda_S$, there is a vector space isomorphism \[ \Hom_{\catO_{S}}(M_S(\mu), M_S(\lambda)) \cong \Hom_{\catO_{S}}(M_S(\lambda'), M_S(\mu')). \]
   
   \end{corollary}
  
     \section{Truncations and reduction rules}\label{Sec5}
    In this section, we introduce a finite truncation of $\catO_{S, \Lambda}$. Then the main results in this paper will be given, i.e., a reduction rule and several applications. 
    
    Fix a subset $S \subset I$ of finite type.  The Ringel duality between truncated subcategories for the parabolic category of affine Kac-Moody algebras and the corresponding quotient induced by the tilting functor $t_S$ is considered in \cite{SVV}. We next consider the duality for a general parabolic category $\catO_{S}$. 
    
    For a finite dimensional associative algebra $A$ over $\bbC$, we abbreviate $A\text{-}\mathrm{mod}$ for the category of all finitely generated left $A$-modules, similarly, $\mathrm{Mod}\text{-}A$ and $\mathrm{mod}\text{-}A$ for the category of right $A$-modules and the category of  finitely generated right $A$-modules respectively.
    
    Let $Q$ be the root lattice of the pair $(\frakg, \frakh)$ and $Q^+$ its submonoid generated by root bases, i.e., $Q^+=\sum_i \bbZ_{\geq 0}\alpha_i$. We say a subset $\Theta \subset \frakh^*$ is \textit{lower-bounded} (resp. \textit{upper-bounded}) if there exist finitely many weights $\mu_j\in \frakh^*$ such that $\Theta$ is contained in the union of the sets $\mu_j + Q^+$ (resp. $\mu_j - Q^+$).
    
    Now fix a lower-bounded equivalence class $\Lambda \in P_S^+/\sim_S$ (it does exist since any equivalence class in $P_S^+/\sim_S$ is a subset of a certain equivalence class in $\frakh^*/\sim$). A \textit{finite truncation} of $\Lambda$ is a subset $\Lambda_b$ of $\Lambda$ satisfying that $ \Lambda_b$ is upper-bounded and $\Lambda_b = \cup_{\mu \in  \Lambda_b}\{\gamma \in \Lambda ; \gamma \leq \mu \}$. Consequently, $\Lambda_b$ is a finite poset with $\leq$.
    
    As in \cite{SVV} we consider a Serre subcategory of $\catO_{S, \Lambda}$ which is a highest weight category possessing finitely many simple objects.  Let $\catO_{S, \Lambda_b}$ be the subcategory of $\catO_{S, \Lambda}$ generated by all simple objects $L(\mu)$ with $\mu \in \Lambda_b$, i.e., for any short exact sequence 
    \[0 \rightarrow M_1 \rightarrow M_2 \rightarrow M_3 \rightarrow 0 \]
    in $\catO_{S, \Lambda}$, $M_2 \in \catO_{S, \Lambda_b}$ if and only if $M_1, M_3 \in \catO_{S, \Lambda_b}$. Denote $\calM_{S, \Lambda_b}\subset \catO_{S, \Lambda_b}$  the subcategory of all modules lying in $\calM_{S, \Lambda}$. Note that all tilting modules  $T_S(\mu)$  and $M_S(\mu)$ for $\mu \in \Lambda_b$ are contained in $\calM_{S, \Lambda_b}$.
    
    Let $\calC \subset \catO_{S, \Lambda_b}$ be the full subcategory of all finitely generated modules. Then $\calM_{S, \Lambda_b} \subset \calC$ and  $\calC$ is a highest weight category in the sense of  \cite[Definition 3.1]{CPS} or \cite[Section 2.4]{SVV}. Taking the minimal projective generator in $\calC$, consider the opposite of its endomorphism ring, denoted by $A$. Then $A$ is a finite dimensional quasi-hereditary algebra and  $\calC$ is equivalent to the category of finitely generated modules of $A$. We call the module $T := \oplus_{\mu \in \Lambda_b}T_S(\mu)$ the \textit{characteristic tilting module} in $\calC$. Set $\calR(A):= \End_\calC(T)$, which is called the \textit{Ringel dual} of $A$.   $\calR(A)$ is also a quasi-hereditary algebra. Define a contravariant functor $\calR_A:=\Hom_A(-, T): A \text{-}\mathrm{mod} \rightarrow  \mathrm{mod} \text{-}\calR(A)$. 
    
    We denote $ \Lambda' : = w_S \cdot (-2\rho - \Lambda)$ and $\Lambda_b' : = w_S \cdot (-2\rho - \Lambda_b)$.  Each simple module $L(\mu)$ for $\mu \in \Lambda_b'$ has a projective cover $P_S(\mu)$ in $\catO_{S, \Lambda'}$. Define $P:= \oplus_{\mu \in \Lambda_b'}P_S(\mu)$. Then by tilting equivalence in Subsection \ref{subsection42}, $ \End_\frakg(P) \cong  \calR(A)$ as algebras. Consider the functor
    \[F:= \Hom_{\catO_S}(P, -): \catO_{S, \Lambda'} \longrightarrow \Mod\text{-}\calR(A). \]
    Then we can easily deduce the following.
    \begin{proposition}
    	Let $F$ be the functor defined as above.
    	\begin{enumerate}
    		\item $F$ is exact and essentially surjective.
    		\item The kernel category $ \Ker(F) $ of $F$ is the (full) subcategory generated by all simple objects $L(\gamma)$ for $\gamma \in \Lambda' \setminus \Lambda_b'$.
    	\end{enumerate}
    \end{proposition}

    Denote $\catO_{S}(\Lambda_b')$ the quotient category of $\catO_{S, \Lambda'}$ by the kernel category $ \Ker(F) $. Therefore, $F$ induces an equivalence of categories 
    $ \catO_{S}(\Lambda_b') \cong \Mod\text{-}\calR(A).  $ Denote $\calD \subset \catO_{S}(\Lambda_b')$  the full subcategory of all finitely generated objects.
    
    Therefore, we build a Ringel duality between $\calC$ and $\calD$, namely, we have the following 
    \[\calC \stackrel{\sim}{\rightarrow} A \text{-}\mathrm{mod} \stackrel{\calR_A}{\longrightarrow} \mathrm{mod} \text{-}\calR(A) \stackrel{\sim}{\leftarrow} \calD,\]
    where $\calR_A = \Hom_A(-, T)$.
    
    The Ringel dual functor from $\calC$ to $\calD$ restricts to an equivalence between the exact category $\calM_{S, \Lambda_b}$ and the exact category of objects in $\calD$ with standard flag. Each simple object $L(\gamma)$ for $\gamma \in \Lambda_b'$ lies in $\calD$. If we denote the standard object associated to $L(\gamma)$ in $\calD$ by $A(\gamma)$ for $\gamma \in \Lambda_b'$, then we get the following property.
    \begin{theorem}\label{thmredu}
    	For $\lambda, \mu \in \Lambda_b$,  let $\lambda'= w_S\cdot(-2\rho-\lambda), \mu'=w_S\cdot (-2\rho- \mu) \in \Lambda'_b$. Then there is a vector space isomorphism  
    	\[  \Hom_{\catO_S}(M_S(\mu), M_S(\lambda)) \cong \Hom_{\calD}(A(\lambda'), A(\mu')).\]
    \end{theorem}	
    \begin{proof}
    	By \cite[Section 2.4]{CM}, we have the contravariant functor $\calR_A: A \text{-}\mathrm{mod} \rightarrow  \mathrm{mod} \text{-}\calR(A)$ restricts to an equivalence between the exact category of left $A$-modules with standard flag and the exact category of right $\calR(A)$-modules with standard flag. This gives an equivalence between $\calM_{S, \Lambda_b}$ and the exact category of objects in $\calD$ with standard flag. Hence the corresponding vector space isomorphism holds, and the theorem is proved as desired.
    \end{proof}

     For any $\frakg$-module $M \in \catO_S$ and $\mu \in \frakh^*$, a weight vector $v$ in $M_\mu$ is called a \textit{maximal (weight) vector} if the space $\frakn_I.v=0$.  
     Then   $\Hom_\frakg(M_S(\mu), M)$ as a vector space is obviously isomorphic to the space of  maximal  vectors $v$ of weight $\mu$ in $M$, denote this space by $M_\mu^{\frakn_I}$.
     
     \begin{proposition}
     	 Fix weights $\lambda, \mu \in P_S^+$, $\lambda \neq \mu$. If $M_S(\lambda)$ has a nonzero maximal vector of weight $\mu$, then $\lambda \geq \mu$ and  $\mathrm{ht}_S(\lambda - \mu) \geq 1$.
     \end{proposition}
    \begin{proof}
    	$\lambda \geq \mu$ is clear. It suffices to show that $\mathrm{ht}_S(\lambda - \mu) \geq 1$. Otherwise, $\mathrm{ht}_S(\lambda - \mu) = 0$. Let $v$ be such a maximal vector of weight $\mu$ in $M_S(\lambda)$, then $v \in \calU(\frakl_S).v_{\lambda}$, where $v_\lambda$ is a highest weight vector in $M_S(\lambda)$ of weight $\lambda$. Note that $\calU(\frakl_S).v_{\lambda}$ is an $\frakl_S$-submodule in $M_S(\lambda)$ of highest weight $\lambda$, hence $\calU(\frakl_S).v_{\lambda} \cong L_S(\lambda)$ as $\frakl_S$-modules. However, $L_S(\lambda)$ has no non-trivial maximal vector of weight $\mu \neq \lambda$. It is a contradiction.
    \end{proof}	

   Next we show that there is a reduction rule of hom-spaces between parabolic Verma modules as in \cite[Lemma 3.11]{WX}, but the proof will be a little different due to the possibility of infinitely many roots involved.  Fix a subset $S$ in $I$, and weights $\lambda, \mu \in P_S^+$. If $\eta = \lambda - \mu= \sum_{i \in I}k_i \alpha_i \in Q$, then we can define a subset $J:= \mathrm{supp}(\eta)=\{i \in I ;  k_i \neq 0 \}$ of $I$ and $S' :=S \cap J$.
     Recall the subalgebras $\frakl_{J}$	and $\frakp_S$ of $\frakg$ defined in the section \ref{sec2},
     \[\frakl_{J} = \frakh \oplus (\bigoplus_{\alpha \in \Delta_J}\frakg_\alpha), \quad  \frakp_S = \frakl_S \oplus (\bigoplus_{\alpha \in \Delta^+ \setminus \Delta^+_S}\frakg_\alpha). \]
      Consider the subalgebra $\frakp_S \cap \frakl_{J}$, denoted by $\frakq_{S'}$. It is a parabolic subalgebra of $\frakl_{J}$ associated to $S'$. Note that 
     $\frakq_{S'}=\frakl_{S'} \oplus \fraku_{J \backslash S'}$,
     where the subalgebra $\fraku_{J \backslash S'}$ is defined as  
     \begin{align}
     \fraku_{J \backslash S'} = \bigoplus_{\alpha \in \Delta_J^+ \backslash \Delta_{S'}^+} \frakg_{\alpha}.
     \end{align}
      Obviously, $\frakq_{S'}$ is a subalgebra of $\frakp_S$. 
     
     Let $L_S(\lambda)$ (resp. $L_{S'}(\lambda)$) be the simple, locally $\frakn_I$-finite (resp. locally $\frakn_{J}$-finite), weight module with highest weight $\lambda$ of $\frakp_S$ (resp. $\frakq_{S'}$); essentially, they are  also simple weight modules of $\frakl_S$ and $\frakl_{S'}$ respectively, hence such notations $L_S(\lambda), L_{S'}(\lambda)$ make sense. There exists a natural injective $\frakq_{S'}$-module homomorphism from $L_{S'}(\lambda)$ to $L_S(\lambda)$ since $L_S(\lambda)$ is completely reducible when it is restricted as a  $\frakq_{S'}$-module. Similarly, we set
      \[ V_{S'}(\lambda) := \calU(\frakl_{J}) \otimes_{\calU(\frakq_{S'})}L_{S'}(\lambda), \] 
     a parabolic Verma module of $\frakl_{J}$ in terms of $S'$. If $S'$ is empty, $ V(\lambda)= V_{\emptyset}(\lambda) $ is just a Verma module of highest weight $\lambda$. 
     
     Since $\Delta_{J}^+ \setminus \Delta_{S'}^+ $ is contained in $ \Delta^+ \setminus \Delta_S^+$, we have the induced $\frakl_{J}$-module homomorphism $V_{S'}(\lambda) \rightarrow M_S(\lambda)$ is injective. Therefore we may regard $V_{S'}(\lambda)$ as an $\frakl_{J}$-submodule of $M_S(\lambda)$.
     
     \begin{theorem}\label{Lmm}
       For any subset $S$ of $I$, and any weights $\lambda, \mu$ in $P_S^+$,  the subsets $S' \subset J$ of $I$ are defined as above. Then we have a vector space isomorphism 
       	\[\mathrm{Hom}_{\catO_S(\frakg)}(M_S(\mu),M_S(\lambda)) \cong \mathrm{Hom}_{\catO_{S'}(\frakl_{J})}(V_{S'}(\mu),V_{S'}(\lambda)). \]
     \end{theorem}
     \begin{proof}
     	 We only need to show that 
     \begin{align}\label{tocheck}
     	M_S(\lambda)^{\frakn_{I}}_\mu \cong V_{S'}(\lambda)^{\frakn_{J}}_\mu. 
     \end{align}
     Let $v_\lambda$ be a (nonzero) maximal vector of weight $\lambda$ in $V_{S'}(\lambda)$, so $v_\lambda$ is also a maximal vector in $M_S(\lambda)$.
     Let 
     \[\underline{\beta}= (\beta_1, \beta_2, \cdots) \quad (\text{resp.} \underline{\gamma}= ( \gamma_1, \gamma_2,\cdots) ) \]
     be a (possibly infinite) sequence of all positive roots lying in $\Delta_S^+$ (resp. $\Delta^+ \setminus \Delta_S^+$) with no repeated terms. Denote $\underline{k}= (k_1, k_2, \cdots)$ a sequence of nonnegative integers in which only finite many terms are nonzero. Then we can define an operation $\underline{k}\cdot \underline{\beta}= \sum_{i}k_i\beta_i$. Note that $M_S(\lambda) = \calU(\frakg).v_\lambda$ and the PBW basis theorem implies that 
     \[\calU(\frakg) = \calU(\fraku_S^-) \otimes_\bbC \calU(\frakp_S), \quad \calU(\frakp_S) = \calU(\frakl_S)\otimes_\bbC \calU(\fraku_S),\]
     and 
     \[ \calU(\frakl_S) = \calU(\frakn_S^-)\otimes_\bbC \calU(\frakh)\otimes_\bbC \calU(\frakn_S) \]
     as vector spaces. Therefore, for any nozero vector $v \in M_S(\lambda)_\mu$, we have $v = u.v_\lambda$ for some $u \in \calU(\fraku_S^-) \otimes_\bbC \calU(\frakn_S^-)$, since $\fraku_S$ and $\frakn_S$ acts on $v_\lambda$ trivially, and $\calU(\frakh)$ acts on $v_\lambda$ as scalars. Moreover, the element $u$ can be chosen in the $(\mu - \lambda )$-weight space of  $ \calU(\fraku_S^-) \otimes_\bbC \calU(\frakn_S^-) $. Then we have 
     \begin{equation}\label{equMu43}
     	M_S(\lambda)_\mu= \sum_{\underline{k}, \underline{l}}(\calU(\fraku_S^-)_{-\underline{l}\cdot \underline{\gamma}} \otimes_\bbC \calU(\frakn_S^-)_{-\underline{k}\cdot \underline{\beta}}).v_\lambda, 
     \end{equation}
     where $\underline{k}, \underline{l}$ are taken over all such sequences such that $ \underline{k}\cdot \underline{\beta}+\underline{l}\cdot \underline{\gamma} = \lambda-\mu$. 
     
     Set $\eta = \lambda - \mu$. If $M_S(\lambda)_\mu$ is not zero, then $\eta \in \bbZ\Delta_J$, where $J = \mathrm{supp}(\eta)$. In the equation (\ref{equMu43}), we can restrict our attention to consider only the terms of $\underline{\beta}$ and $\underline{\gamma}$  in $\Delta_{S'}^+$ and $\Delta^+_{J}\setminus \Delta_{S'}^+ $ respectively. Hence
      \[ \calU(\frakn_S^-)_{-\underline{k}\cdot \underline{\beta}}= \calU(\frakn_{S'}^-)_{-\underline{k}\cdot \underline{\beta}} \quad \text{ and } \quad \calU(\fraku_S^-)_{-\underline{l}\cdot \underline{\gamma}}=\calU(\fraku_{J \setminus S'}^-)_{-\underline{l}\cdot \underline{\gamma}}\]
      in the equation (\ref{equMu43}).
      
     Similarly, we have  
     \[V_{S'}(\lambda)_\mu= \sum_{\underline{k}, \underline{l}}(\calU(\fraku_{J\setminus S'}^-)_{-\underline{l}\cdot \underline{\gamma}} \otimes_\bbC \calU(\frakn_{S'}^-)_{-\underline{k}\cdot \underline{\beta}}).v_\lambda, \]
     for the sequences $\underline{\beta},  \underline{\gamma}$ with all terms lying in $\Delta_{S'}^+$ and $\Delta^+_{J}\setminus \Delta_{S'}^+$, and all sequences  $\underline{k}, \underline{l}$ defined above satisfying $ \underline{k}\cdot \underline{\beta}+\underline{l}\cdot \underline{\gamma} = \eta $. Hence $M_S(\lambda)_\mu=V_{S'}(\lambda)_\mu$.
     
     Obviously, $ M_S(\lambda)^{\frakn_{I}}_\mu $ is contained in $ V_{S'}(\lambda)^{\frakn_{J}}_\mu$ since $\frakn_{J} \subset \frakn_I$. On the other hand, note that for each $i \in I \setminus J$, the nonzero root vector $e_i$ in the space $\frakg_{\alpha_i}$ commutes with $\frakn_{S'}^-$ and $\fraku_{J \setminus S'}^-$, thus $e_i$ kills all elements in $V_{S'}(\lambda)$. In particular, we get the isomorphism (\ref{tocheck}) which implies this theorem. 
     \end{proof}
 In particular, if $J \cap S =\emptyset$, then we have a vector space isomorphism 
\begin{equation}\label{key}
	 \mathrm{Hom}_{\catO_S(\frakg)}(M_S(\mu),M_S(\lambda)) \cong \mathrm{Hom}_{\catO(\frakl_{J})}(V(\mu),V(\lambda)).
\end{equation}
Therefore, if $\frakl_{J}$ is further a finite dimensional reductive Lie algebra, then 
\[  \Dim_\bbC \mathrm{Hom}_{\catO_S(\frakg)}(M_S(\mu),M_S(\lambda)) \leq 1. \] 
If $\frakl_{J}$ is not finite dimensional, we next point out that there exists the unique embedding  property for certain cases.

 For $\Lambda \in \frakh^*/\sim$ of positive level, $\lambda, \mu \in \Lambda$, there exists the unique embedding theorem for Verma modules (cf. \cite[Theorem 2.5.3]{KT1}). Applying the tilting functor, we have the following. 
  \begin{proposition} \label{Prop65}
 	Let $\Lambda \in \frakh^*/\sim$ be of positive or negative level. For any $\lambda, \mu \in \Lambda$, we have
 	\[\mathrm{dim}_\bbC  \mathrm{Hom}_\catO(M(\mu), M(\lambda)) \leq 1. \]

 \end{proposition}
 \begin{proof}
 	If $\Lambda$ is of negative level, then $-2\rho-\Lambda$ is of positive level. By Corollary ~\ref{cor3.8}, for any weights $\lambda, \mu \in \Lambda$, the tilting functor $t$ gives a vector space isomorphism between $ \mathrm{Hom}_\catO(M(\mu), M(\lambda)) $ and $ \mathrm{Hom}_\catO(M(-2\rho-\lambda), M(-2\rho-\mu)) $, then by  \cite[Theorem 2.5.3]{KT1} we prove it as desired.
 \end{proof}	
 
 \begin{remark}
 In general, the unique embedding property fails for infinite dimensional Kac-Moody algebra. For example,  the hom-space of Verma modules at critical level may have dimension greater than 1 (see, e.g., \cite{Wallach,AF,KK}). 
 \end{remark}

     Fix a subset $J \subset I$. Note that $\frakl_J$ is a Kac-Moody algebra essentially as showed in Section \ref{sec2}. Therefore, we say $\lambda \in \frakh^*$ is \textit{$J$-dominant (resp. $J$-antidominant)} if $\langle \lambda + \rho, \alpha_i^\vee\rangle \notin \bbZ_{< 0} $ $( \text{ resp.} \notin \bbZ_{> 0}), \forall i \in J$. Similarly, we can also consider the \emph{$J$-positive level} and \emph{$J$-negative level} as the definitions in Subsection \ref{subsection4.1}. Moreover, we say a weight $\lambda \in \frakh^*$ is of ($J$-)positive (resp. negative) level if $\lambda$ lies in an equivalence class of ($J$-)positive (resp. negative) level. Then we have
     
     \begin{theorem} \label{cor7.3}
     	For $\lambda, \mu \in P_S^+$, let $J:=\mathrm{supp}(\lambda - \mu) \subset I$. If the height and the $S$-height of $\lambda-\mu$ are equal, i.e.,  $\mathrm{ht}_S(\lambda -\mu)= \mathrm{ht}(\lambda - \mu)$ and $\lambda$ is of $J$-positive level (or of $J$-negative level), then 
     	\begin{align}\label{ineq}
     	\mathrm{dim}_\bbC \mathrm{Hom}_{\catO_S}(M_S(\mu), M_S(\lambda)) \leq 1. 
     	\end{align}
     \end{theorem}	
     \begin{proof}
     	The condition $\mathrm{ht}_S(\lambda -\mu)= \mathrm{ht}(\lambda - \mu)$ implies that $J \cap S=\emptyset$. Then by the isomorphism (\ref{key}) and  Proposition \ref{Prop65} we get this inequality (\ref{ineq}).
     \end{proof}
 
     \begin{corollary}
 	    For an integral dominant weight $\lambda \in \frakh^*$, and any element $w$ in the parabolic subgroup $W_{I\setminus S}$, we have 
 	    \[\mathrm{dim}_\bbC \mathrm{Hom}_{\catO_S}(M_S(w\cdot\lambda), M_S(\lambda)) \leq 1 \]
     \end{corollary}
  \begin{proof}
  	 It follows from Theorem \ref{cor7.3} and $\mathrm{ht}_S(\lambda -w\cdot\lambda)= \mathrm{ht}(\lambda - w\cdot\lambda)$ for all elements $w \in W_{I\setminus S}$.
  \end{proof}	

\section*{Acknowledgments}

This project is partially supported by the NSF of China (Grants 11771410 and 11931009). It was done during the author's visit to the Kansas State University. The author would like to thank Prof. Z. Lin for his warm hospitality and careful guidance. The author also thanks his advisors Prof. Y. Gao and Prof. H. Chen for a lot of helpful discussions during the preparation of the paper. The author is also grateful to the referees for their valuable suggestions to improve the paper.
      
	\bibliographystyle{amsplain}

\end{document}